\documentclass[10pt]{article}

\usepackage{amsmath}    % need for subequations
\usepackage{graphicx}   % need for figures
\usepackage{verbatim}   % useful for program listings
\usepackage{color}      % use if color is used in text

\usepackage{amssymb}
\usepackage{amsthm}

\usepackage[top=2cm, bottom=2cm, left=2cm, right=2cm]{geometry}

\theoremstyle{theorem}
\newtheorem{Def}{Definition}[section]

\newtheorem{Lem}[Def]{Lemma}
\newtheorem{Thm}[Def]{Theorem}

\newtheorem{Ass}[Def]{Assumption}
\theoremstyle{definition}
\newtheorem{Rem}[Def]{Remark}

\newcommand{\e}{\mathbb{E}}
\newcommand{\real}{\mathbb{R}}
\newcommand{\n}{\mathbb{N}}

\newcommand{\1}{{\bf 1}}

\newcommand{\F}{\mathcal{F}}

\newcommand{\ff}{\mathbb{F}}

\allowdisplaybreaks

\begin{document}

\title{
%title
On the Euler-Maruyama scheme for spectrally one-sided L\'evy driven SDEs with H\"older continuous coefficients\\
%On the Euler scheme for L\'evy driven SDEs with H\"older continuous coefficients via Wiener-Hopf factorization
%Euler-Poisson scheme for stochastic differential equations driven by L\'evy process with H\"older continuous coefficients
	}
\vskip 40pt	
\author{Libo Li\footnote{Corresponding Author}\\
		School of Mathematics and Statistics,\\
		University of New South Wales,\\
		NSW 2006, Australia,\\
		Email: libo.li@unsw.edu.au\\
\and
 Dai Taguchi\\
		Graduate School of Engineering Science,\\
		Osaka University,\\
		1-3, Machikaneyama-cho, Toyonaka,\\
		Osaka, Japan, \\
		Email: dai.taguchi.dai@gmail.com,	\\	
}

\date{}
\maketitle
\begin{abstract}
We study in this article the strong rate of convergence of the Euler-Maruyama scheme and associated with the jump-type equation introduced in Li and Mytnik \cite{LiMy}. We obtain the strong rate of convergence under similar assumptions for strong existence and pathwise uniqueness.
Models of this type can be considered as a generalization of the CIR (Cox-Ingersoll-Ross) process with jumps.
%60H35 Computational methods for stochastic equations
%41A25 Rate of convergence, degree of approximation
%60H10 Stochastic ordinary differential equations
%65C30 Stochastic differential and integral equations
\end{abstract}
%\vfill\break 
%\tableofcontents
%\vfill\break 
\section{Introduction}
In mathematical finance, a popular model for short term interest rates is the Cox-Ingersoll-Ross (CIR) model, which is the solution to the one-dimensional stochastic differential equation (SDE)
\begin{align}\label{SDE-cir}
X_t=
x_0
+\int_{0}^{t} a(c-X_s) ds
+\int_{0}^{t} \sqrt{X_{s}}dW_s\quad 
~x_0 \in \real,
~t \in [0,T],
\end{align}
where $a,c> 0$ and $W=(W_t)_{0\leq t \leq T}$ is a standard one-dimensional Brownian motion. There has been a push in the financial mathematics literature to generalize the CIR models to include jumps. The most noteable works in this direction are the affline jump-diffusion models proposed in Duffie et al. \cite{DFS, DPS}.
%
%However in the case of Komatsu \cite{Komatsu} that pathwise uniqueness holds for a symmetric $\alpha$-stable driven SDE with $1/\alpha$-H\"older continuous coefficients only when there is no drift term (for $\alpha \in (1,2)$). For existence and pathwise uniqueness of stable driven SDEs, we mention also the works of Priola \cite {PE} and Fournier \cite{F} and the references within. 

%\begin{gather}
%X_t = x_0 +\int_{0}^{t} h(X_{s-})dZ_s,\quad 
%~x_0 \in \real,
%~t \in [0,T],\label{SDE_2}
%\end{gather}
%where the driving process $Z$ is a symmetric $\alpha$-stable process with $\alpha \in (1,2)$ with the jump coefficient $h$ been $1/\alpha$-H\"older continuous. 

Motivated by the recent developments of continuous-state branching processes. It was shown in Fu and Li \cite{FL}, and later extended in Li and Mytnik \cite{LiMy} to more general jump type equations, that is if $b$, $\sigma$ and $h$ are H\"older continuous and $h$ is non-decreasing then existence and pathwise uniqueness of solution holds for SDEs of the form
\begin{align}\label{SDE_1}
X_t& = x_0
+\int_{0}^{t} b(X_{s-})ds
+\int_{0}^{t} \sigma(X_{s-})dW_s
+\int_{0}^{t} h(X_{s-})dL_s,\quad 
~x_0 \in \real,
~t \in [0,T].\\
L_t & =\int_{0}^{t} \int_{0}^{\infty} z \widetilde{N}(ds,dz).\label{ll2}
\end{align}
where the process $W=(W_t)_{0\leq t \leq T}$ is a standard one-dimensional Brownian motion and $\widetilde N$ is a compensated Poisson random measure with intensity or L\'evy measure $\nu$ satisfying the condition $\int^\infty_0 \{z^2\wedge z\}\, \nu(dz) < \infty$. In the recent paper of Jiao et al. \cite{JMS, JMSS}, in order to capture the persistency of low interest rate, self-exciting and large jump behaviours exhibited by sovereign interest rates and power markets, a version of the model considered in \cite{FL, LiMy} was introduced to the financial mathematics literature as the $\alpha$-CIR process. 

In pracitice, the solution to equation (\ref{SDE_1}) is rarely analytically tractable, the goal of this article is to study under similar assumptions to those of \cite{LiMy}, the strong rate of convergence for Euler-Maruyama scheme associated with the SDE \eqref{SDE_1}. From the point of view of strong existence and pathwise uniqueness of a solution, the fact that the L\'evy measure $\nu$ is stable plays (as chosen in Jiao et al. \cite{JMS}) very little role (see Theorem 2.3 in \cite{LiMy}). One can consider any spectrally positive L\'evy process of the form given in \eqref{ll2} and produce a wide range of {\it generalized CIR processes} with different jump structures. 

Given $n\in \mathbb{N}$ and a time grid $0=t_0<t_1\dots< t_n = T$, the Euler-Maruyama scheme associated with equation \eqref{SDE_1} is given by $X_0 := x_0$ and
\begin{align*}
X^{(n)} _{t_i} := x_0 +\int_{(0,t_i]} \sum_{j=0}^{n-1} b(X^{(n)} _{t_j})\1_{(t_j,t_{j+1}]}(s) ds + \int_{(0,t_i]} \sum_{j=0}^{n-1} \sigma(X^{(n)} _{t_j})\1_{(t_j,t_{j+1}]}(s) dW_s + \int_{(0,t_i]} \sum_{j=0}^{n-1} h(X^{(n)} _{t_j})\1_{(t_j,t_{j+1} ]}(s)  dL_s
\end{align*}
and one can extend the definition of the Euler-Maruyama scheme to continuous time by setting 
\begin{align*}
X^{(n)}_{t} = x_0 + \int_{(0,t]} b(X^{(n)} _{\eta_n(s)})ds+ \int_{(0,t]} \sigma(X^{(n)} _{\eta_n(s)})dW_s + \int_{(0,t]} h(X^{(n)} _{\eta_n(s)})dL_s
\end{align*}
where $\eta_n(s):=t_j$ if $s\in (t_j,t_{j+1}]$. The process $(X^{(n)}_{\eta_n(t)})_{0\leq t\leq T}$ is left continuous and for the purpose of this paper, we take equally spaced time grid of size $T/n$.

Using techniques from Yamada and Watanabe \cite{YaWa}, Gy\"ongy and R\'asonyi \cite{GyRa} proved that if the drift coefficient $b$ is the sum of a Lipschitz and a non-increasing $\rho$-H\"older continuous, the diffusion coefficient $\sigma$ is $\gamma$-H\"older continuous with $\gamma \in [1/2,1]$ and the jump coefficient $h=0$, then
\begin{align*}%\label{EM_BM}
	\e[|X_T-X_{T}^{(n)}|]
	\leq \left\{ \begin{array}{ll}
	Cn^{-\frac{\rho}{2} \wedge (\gamma-\frac{1}{2})} &\text{ if } \gamma \in (1/2,1],\\
	C(\log n)^{-1}, &\text{ if } \gamma =1/2.
	\end{array}\right.
\end{align*}
In \cite{Y}, Yan proved similar results when $\gamma>1/2$ by using Tanaka's formula. These results are later extended, in for exmaple \cite{MeTa, NT}, to SDEs with irregular drift and diffusion coefficients. In the case where $h\neq 0$, $L$ is a symmetric $\alpha$-stable process with $\alpha \in (1,2)$ and $b=\sigma=0$, Hashimoto and Tsuchiya \cite{HaTsu} shown using the method of Komatsu \cite{Komatsu}, if the coefficient jump $h$ is bounded $\gamma$-H\"older continuous with $\gamma \in [1/\alpha,1]$, then
\begin{align*}%\label{EM_stable}
\e[|X_T-X_{T}^{(n)}|^{\alpha-1}]
\leq \left\{ \begin{array}{ll}
Cn^{-(\gamma-\frac{1}{\alpha})} &\text{ if } \gamma \in (1/\alpha,1],\\
C(\log n)^{-(\alpha-1)}, &\text{ if } \gamma =1/\alpha.
\end{array}\right.
\end{align*}
We mention here also the works of Hashimoto \cite{Ha}, Mikulevi\v cius and Xu \cite{MiXu},  Qiao \cite{Qiao} for strong convergence and Mikulevi\v cius and Zhang \cite{MiZh} for weak convergence. However there is little in the current literature on the Euler-Maruyama scheme for jump-type equation with H\"older continuous coefficients and drift. To the best of our knowledge, there is no result on the strong rate of convergence for equation of the form \eqref{SDE_1}. 

%Most likely, the main reason is that for a generic L\'evy process, it is usually not so straight forward to simulate it's increments and therefore difficult to apply the standard Euler-Maruyam scheme. 

The structure of the current work is as follows. In section \ref{notation} we introduce the necessary notations and our standing assumptions. In section \ref{YW}, we introduce the Yamada-Watanabe approximation technique and give two auxiliary results in Lemma \ref{key_lem_0} and Lemma \ref{key_lem12}, which are used in controlling the jump part of the approximation. In section \ref{L1}, under boundedness assumption on the coefficients $\sigma$ and $h$, we obtain in Theorem \ref{main_1} the strong rate of convergence of the Euler-Maruyama scheme for driving L\'evy processes which are non-square integrable. In section \ref{L2}, we consider the case of square integrable L\'evy processes and obtain in Theorem \ref{l2t} the strong rate of convergence without any boundedness assumption on the coefficients.

\subsection{Notations and Assumptions}\label{notation}
We work on the usual filtered probability space $(\Omega, \F, \mathbb{P})$ endowed with a filtration $\ff:=(\F_t)_{t\geq 0}$ which satisfies the usual conditions and $\F_\infty \subset \F$. We denote the sup-norm by $\| \cdot \|_\infty$ and set
\begin{align*}
\alpha_{\nu}:=\inf\{\widehat{\alpha}>1; \lim_{x \to 0+} x^{\widehat{\alpha}-1} \int_{x}^{\infty}z \nu(dz)=0\}.
\end{align*}

%The L\'evy measure $\nu$ and the coefficients $b$, $\sigma$ and $h$ are assumed to satisfy the following assumptions.
\begin{Ass}\label{Ass_1}
We assume that the L\'evy measure $\nu$, and the coefficients $b$, $\sigma$ and $h$ satisfies the following conditions:
\begin{itemize}	
	\item[(i)]
	There exist $\zeta \in [1/2,1]$ and $K_0>0$ such that 
	\begin{align*}
		\sup_{t,s \in [0,T]}\e[|L_t-L_s|] \leq K_0|t-s|^{\zeta}.
		%\text{ and } \e[|L_t|^{?}] \leq K_0.
	\end{align*}
Note that examples of $L$ include compensated $\alpha$-stable L\'evy process for $\alpha \in [1,2]$, compensated square integrable L\'evy processes and compensated compound Poisson process with integrable jump size.
	
	\item[(ii)] The L\'evy measure $\nu$ is such that $\nu((-\infty,0))=0$ and $\displaystyle \int_{0}^{\infty} \{z \wedge z^2\} \nu(dz)<\infty$.	
	\item[(iii)] The drift coefficient $b$ is of the form $b=b_1+b_2$ where $b_1$ is a Lipschitz continuous function, and $b_2$ is a non-increasing $\rho$-H\"older continuous function with $\rho \in (0,1)$, that is,
	\begin{align*}
	K_1:=\sup_{x,y \in \real,x\neq y} \frac{|b_1(x)-b_1(y)|}{|x-y|}
	+\sup_{x,y \in \real,x\neq y} \frac{|b_2(x)-b_2(y)|}{|x-y|^{\rho}}
	<\infty.
	\end{align*}
	
	%\item[(i)] $b$ is a one-sided Lipschitz and linear growth function, that is, there exists $K>0$ such that for any $x,y \in \real$,
	%\begin{align*}
	%(x-y)(b(x)-b(y)) \leq K |x-y|^2, \text{ and } |b(x)|\leq K_1(1+|x|).
	%\end{align*}
		
	\item[(iv)] The diffusion coefficient $\sigma$ is an $\gamma$-H\"older continuous function with $\gamma \in [1/2,1)$ and the coefficient $h$ is an $\beta$-H\"older continuous function with $\beta \in (1-1/\alpha_{\nu},1)$, that is,
	\begin{align*}
	K_2:=\sup_{x,y \in \real,x\neq y} \frac{|\sigma(x)-\sigma(y)|}{|x-y|^{\gamma}}
	+\sup_{x,y \in \real,x\neq y} \frac{|h(x)-h(y)|}{|x-y|^{\beta}}
	<\infty.
	\end{align*}
	
	\item[(v)] The coefficient $h$ is a non-decreasing function.
		
\end{itemize}
\end{Ass}

	By Assumption \ref{Ass_1} (iii) and (iv), there exists $K_3$ such that for any $x \in \real$, $|b(x)|+|\sigma(x)|+|h(x)|\leq K_3(1+|x|)$ and we denote $K:=\max\{K_0,K_1,K_2,K_3\}$.

\begin{Rem}\label{remark1.2} We list now some consequences of Assumption \ref{Ass_1}.
\begin{itemize}
	\item[(i)]
	From Lemma 2.1 of Li and Mytnik \cite{LiMy}, if $\int_{0}^{\infty} \{z \wedge z^2\} \nu(dz)<\infty$ then $\alpha_{\nu} \in [1,2]$ and for any $\widehat{\alpha}>\alpha_{\nu}$, $\lim_{x \to 0+} x^{\widehat{\alpha}-2} \int_{0}^{x} z^2 \nu(dz)=0$.
 % and we have $\beta \in (1/2,1)$.	
	\item[(ii)]
From Theorem 25.3 and Theorem 25.18 of Sato \cite{SK}, we know that for any $p>0$, $\mathbb{E}[|L_t|^p]$ and $\mathbb{E}[\sup_{s\leq t} |L_s|^p]$ are finite for all $t \geq 0$ if and only if $\int_1^\infty z^p \nu(dz) < \infty$. 
\end{itemize}
\end{Rem}

%We define the Poisson process $\mathbf{N}(n/T)=(\mathbf{N}_t(n/T))_{t \geq 0}$, where $\mathbf{N}_t(n/T):=\sum_{i=1}^{\infty}\1_{\{\tau_i^{(n)} \leq t\}}$ and 

\subsection{Yamada and Watanabe Approximation Technique}\label{YW}
To deal with the H\"older continuity of the coefficients $\sigma$ and $h$, we introduce below the Yamada and Watanabe approximation technique (see for example \cite{GyRa,LiMy,YaWa}).
For each $\delta \in (1,\infty)$ and $\varepsilon \in (0,1)$, we select a continuous function $\psi _{\delta, \varepsilon}: \real \to \real^+$ with support of $\psi _{\delta, \varepsilon}$ belongs to $[\varepsilon/\delta, \varepsilon]$ and is such that
\begin{align*} 
\int_{\varepsilon/\delta}^{\varepsilon} \psi _{\delta, \varepsilon}(z) dz
= 1 \quad \text{ and } \quad  0 \leq \psi _{\delta, \varepsilon}(z) \leq \frac{2}{z \log \delta}, \:\:\:z > 0.
\end{align*}
%We point out that since $\int_{\varepsilon/\delta}^{\varepsilon} \frac{2}{z \log \delta} dz=2$, such a function $\psi_{\delta, \varepsilon}$ must exist.
We define a function $\phi_{\delta, \varepsilon} \in C^2(\real;\real)$ by setting
\begin{align*}
\phi_{\delta, \varepsilon}(x)&:=\int_0^{|x|}\int_0^y \psi _{\delta, \varepsilon}(z)dzdy.
\end{align*}
It is straight forward to verify that $\phi_{\delta, \varepsilon}$ has the following useful properties: 
\begin{align} 
&|x| \leq \varepsilon + \phi_{\delta, \varepsilon}(x), \text{ for any $x \in \real $}, \label{phi3}\\ 
%&\frac{\phi'_{\delta, \varepsilon}(x)}{x}>0, \text{ for any $x \in \real 
%\setminus \{0\}$}. \label{phi1}\\
&0 \leq |\phi'_{\delta, \varepsilon}(x)| \leq 1, \text{ for any $x \in \real$} \label{phi2}, \\
&\phi'_{\delta, \varepsilon}(x) \geq 0, \text{ for } x\geq 0 \text{ and } \phi'_{\delta, \varepsilon}(x) < 0, \text{ for } x< 0, \label{phi1}\\
&\phi''_{\delta, \varepsilon}(\pm|x|)=\psi_{\delta, \varepsilon}(|x|)
\leq \frac{2}{|x|\log \delta}{\bf 1}_{[\varepsilon/\delta, \varepsilon]}(|x|)
\leq \frac{2\delta }{\varepsilon \log \delta},
\text{ for any $x \in \real \setminus\{0\}$}. \label{phi4}
\end{align}

We present below two auxiliary lemmas, which are used to control the jumps in the estimation of the 	strong error. Lemma \ref{key_lem_0} below is analogues to Lemma 3.2 given in \cite{LiMy}.
\begin{Lem}\label{key_lem_0}
	Suppose that the L\'evy measure $\nu$ satisfies $\int_0^{\infty} \{z \wedge z^2\} \nu(dz) < \infty$.
	Let $\varepsilon \in (0,1)$ and $\delta \in (1,\infty)$.
	Then for any $x \in \real$, $y \in \real \setminus\{0\}$ with $xy \geq 0$ and $u>0$, it holds that
	\begin{align*}%\label{key_lem_1}
	&\int_{0}^{\infty}
	\{\phi_{\delta,\varepsilon}(y+xz)-\phi_{\delta,\varepsilon}(y)-xz\phi_{\delta,\varepsilon}'(y)\} \nu(dz)\notag
	\leq 2 \cdot \1_{(0,\varepsilon]}(|y|)
	\left\{
	\frac{|x|^2}{\log \delta} \left(\frac{1}{|y|} \wedge \frac{\delta}{\varepsilon}\right) \int_{0}^{u} z^2 \nu(dz)
	+ |x| \int_{u}^{\infty} z \nu(dz)
	\right\}.
	\end{align*}
\end{Lem}
\begin{proof}
	Let $x \in \real$, $y \in \real \setminus\{0\}$ with $xy \geq 0$ and $z >0$.
	By the second order Taylor's expansion for $\phi_{\delta,\varepsilon}$, it follows from \eqref{phi4} that
	\begin{align*}
	\phi_{\delta,\varepsilon}(y+xz)-\phi_{\delta,\varepsilon}(y)-xz\phi_{\delta,\varepsilon}'(y)
	&=|xz|^2 \int_{0}^{1} \theta \phi_{\delta,\varepsilon}''(y+ \theta xz)d \theta
	\leq \frac{2|xz|^2}{\log \delta} \int_{0}^{1} \frac{\theta \1_{[\varepsilon/\delta,\varepsilon]}(|y+ \theta xz|) }{|y+ \theta xz|} d \theta.
	\end{align*}
	Since $xy \geq 0$, we have $|y| \leq |y+\theta xz|$ and $\1_{[\varepsilon/\delta,\varepsilon]}(|y+ \theta xz|) \leq \1_{(0,\varepsilon]}(|y|)$.
	Hence we obtain
	\begin{align}\label{key_lem_2}
	\phi_{\delta,\varepsilon}(y+xz)-\phi_{\delta,\varepsilon}(y)-xz\phi_{\delta,\varepsilon}'(y)
	& \leq \frac{2 |xz|^2 \1_{(0,\varepsilon]}(|y|)}{ \log \delta} \left(\frac{1}{|y|} \wedge \frac{\delta}{\varepsilon}\right).
	\end{align}
	Moreover, since $xy \geq 0$, by \eqref{phi1} we have $x \phi_{\delta,\varepsilon}'(y) \geq 0$. This together with the fact that the right hand side of \eqref{key_lem_2} has $\1_{(0,\varepsilon]}(|y|)$, we obtain
	\begin{align}\label{key_lem_3}
	\phi_{\delta,\varepsilon}(y+xz)-\phi_{\delta,\varepsilon}(y)-xz\phi_{\delta,\varepsilon}'(y)
	&\leq \1_{(0,\varepsilon]}(|y|) \{\phi_{\delta,\varepsilon}(y+xz)-\phi_{\delta,\varepsilon}(y) \} \notag\\
	&=\1_{(0,\varepsilon]}(|y|) xz \int_{0}^{1}\phi_{\delta,\varepsilon}'(y+\theta xz) d \theta
	\leq \1_{(0,\varepsilon]}(|y|) |xz|.
	\end{align}
	The result then follows from \eqref{key_lem_2} and \eqref{key_lem_3}.
\end{proof}
%Lemma \ref{key_lem12} below is needed in estimating the jump terms for the Euler-Poisson scheme and Euler-Maruyama scheme.
\begin{Lem}\label{key_lem12}
	Suppose that the L\'evy measure $\nu$ satisfies $\int_0^{\infty} \{z \wedge z^2\} \nu(dz) < \infty$.
	Let $\varepsilon \in (0,1)$ and $\delta \in (1,\infty)$.
	Then for any $x,x' \in \real$, $y \in \real$ and $u \in (0,\infty]$, it holds that
	\begin{align}\label{key_lem_122}
	&\int_{0}^{\infty}\left|
		\phi_{\delta,\varepsilon}(y+xz)-\phi_{\delta,\varepsilon}(y+x'z)-(x-x')z\phi_{\delta,\varepsilon}'(y)
	\right| \nu(dz) \notag\\
	&\leq 2	\left\{ \frac{\delta ( |x-x'|^2+|x'||x-x'|)}{\varepsilon\log \delta}\int_{0}^{u} z^2 \nu(dz)
	+ |x-x'| \int_{u}^{\infty} z \nu(dz)
	\right\}.
	\end{align}
	In particular, if $x'=0$, then
	\begin{align}\label{key_lem_123}
	\int_{0}^{\infty}
	\{\phi_{\delta,\varepsilon}(y+xz)-\phi_{\delta,\varepsilon}(y)-xz\phi_{\delta,\varepsilon}'(y)\} \nu(dz)
	\leq 2	\left\{ \frac{\delta |x|^2}{\varepsilon\log \delta}\int_{0}^{u} z^2 \nu(dz)
	+ |x| \int_{u}^{\infty} z \nu(dz)
	\right\}.
	\end{align}
\end{Lem}
\begin{proof}
	For $z \in (0,u)$, from the second order Taylor's expansion for $\phi_{\delta,\varepsilon}$ and mean value theorem applied to $\phi_{\delta,\varepsilon}'$, we obtain from \eqref{phi4},
	\begin{align*}
	&\left| \phi_{\delta,\varepsilon}(y+xz)-\phi_{\delta,\varepsilon}(y+x'z)-(x-x')z\phi_{\delta,\varepsilon}'(y) \right|\\
	&\leq \left| \phi_{\delta,\varepsilon}(y+xz)-\phi_{\delta,\varepsilon}(y+x'z)-(x-x')z\phi_{\delta,\varepsilon}'(y+x'z) \right|
	+|x-x'||z|\left| \phi_{\delta,\varepsilon}'(y)-\phi_{\delta,\varepsilon}'(y+x'z) \right|\\
	&\leq |x-x'|^2 |z|^2 \int_{0}^{1} \theta \phi_{\delta,\varepsilon}''(y+ \theta xz+(1-\theta)x'z)d \theta
	+|x'| |x-x'||z|^2 \int_{0}^{1} \phi_{\delta,\varepsilon}''(y+ \theta x'z)d \theta\\
	&\leq \left\{ |x-x'|^2 |z|^2+|x'| |x-x'||z|^2 \right\}\frac{2\delta }{\varepsilon \log \delta}.
	\end{align*}	
	For the $z \in [u,\infty)$, apply mean value theorem to $\phi_{\delta, \varepsilon}$, 
	\begin{align*}
	& \left|	\phi_{\delta,\varepsilon}(y+xz)-\phi_{\delta,\varepsilon}(y+x'z)-(x-x')z\phi_{\delta,\varepsilon}'(y)
	\right|\\
	& \qquad  = |x-x'| |z| \int_{0}^{1} \left|\phi'_{\delta,\varepsilon}(y+\theta xz+(1-\theta)x'z)-\phi'_{\delta,\varepsilon}(y) \right| d\theta
	\leq 2 |x-x'| |z|.
	\end{align*}
	%Note that $\phi'$ is increasing since $\phi''\geq 0$. This implies that 
	%if $x>0$, then $0\leq (\phi'_{\delta,\varepsilon}(v^*)-\phi_{\delta,\varepsilon}'(y))$ and if $x\leq 0$, then $(\phi'_{\delta,\varepsilon}(v^*)-\phi_{\delta,\varepsilon}'(y))\leq 0$. 
	%$x \{\phi'_{\delta,\varepsilon}(y+\theta xz)-\phi'_{\delta,\varepsilon}(y) \} \geq 0$.
	%Therefore, by \eqref{phi2}, we have
	%\begin{align*}
	%\phi_{\delta,\varepsilon}(y+xz)-\phi_{\delta,\varepsilon}(y)-xz\phi_{\delta,\varepsilon}'(y)
	%& = |x|z \int_{0}^{1} \left|\phi'_{\delta,\varepsilon}(y+\theta xz)-\phi'_{\delta,\varepsilon}(y) \right| d\theta
	%\leq 2|x|z,
	%\end{align*}
	This concludes the proof of \eqref{key_lem_122}. In the case where $x'=0$, then since $\phi''_{\delta,\varepsilon} \geq 0$, we have
	\begin{align*}
	\phi_{\delta,\varepsilon}(y+xz)-\phi_{\delta,\varepsilon}(y)-xz\phi_{\delta,\varepsilon}'(y)
	&=|xz|^2 \int_{0}^{1} \theta \phi_{\delta,\varepsilon}''(y+ \theta xz)d \theta\
	\geq 0,
	\end{align*}
	which concludes the proof of \eqref{key_lem_123}.
\end{proof}

\vskip 5pt 
\begin{Rem}\label{key_lem13}
Suppose that the L\'evy measure $\nu$ satisfies the condition $\int_1^{\infty} z^2 \nu(dz) < \infty$ then one can take $u = \infty$ in Lemma \ref{key_lem12} and the right hand side of \eqref{key_lem_122} and \eqref{key_lem_123} are still finite.
\end{Rem}

\section{Strong Rate of Convergence}\label{L1}
\subsection{The Non-Square Integrable Case}\label{L1}
In this subsection, we compute the strong rate of convergence in the case where $L$ is a non-square integrable. The typical example one should keep in mind is when the L\'evy measure $\nu$ is spectrally positive $\alpha$-stable with $\alpha \in [1,2]$.

%\begin{proof}
%	We remove the hitting error $n^{-1/4}$ and replace $\log(n)/n$ by $1/n$ for the case $\gamma \in (1/2,1]$.
%\end{proof}

\begin{Lem}\label{EM_esti_1}
	Suppose that Assumption \ref{Ass_1} holds and $h$ is bounded.
	\begin{itemize}
		\item[(i)]
		There exists $C_1>0$ depend on $x_0$, $K$, $T$ and $\|h\|_{\infty}$ such that for any $t \in [0,T]$,
		\begin{align}\label{EM_esti_2}
		\e\big[\, \sup_{t\leq T}|{X}_{t}^{(n)}|\,\big]
		\leq C_1.
		\end{align}
		\item[(ii)]
		There exists $C_2>0$ depend on $x_0$, $C_1$, $K$, $T$ and $\|h\|_{\infty}$ such that for any $t \in [0,T]$,
		\begin{align}
		\e[|{X}_{t}^{(n)}-{X}_{\eta_{n}(t)}^{(n)}|]
		\leq C_2 \left( \frac{1}{n}\right)^{1/2}. \notag
		\end{align}
		
	\end{itemize}
\end{Lem}
\begin{proof}
To prove $(i)$, we aim to apply Lemma 3.2 of Gy\"ongy and R\'asonyi \cite{GyRa}. To bound the stochastic integral against $L$, we note that by Theorem 7.30 of He et al. \cite{HWY}, there exists a localizing sequence of stopping times $(T_m)_{m\in \mathbb{N}}$ with $T_m\uparrow \infty$ such that 
$\int_{0}^{t} h({X}_{\eta_{n}(s)}^{(n)})dL^{T_m}_s  \in \mathcal{H}^1$, where $\mathcal{H}^1$ is the martingale Hardy space. By applying the Burkholder-Davis-Gundy inequality, see for example Theorem 10.36 of He et al. \cite{HWY}, we obtain
\begin{align*}
\e\Big[\,\sup_{t\leq T} \, \Big|\int_{0}^{t} h({X}_{\eta_{n}(s)}^{(n)})dL^{T_m}_s \Big| \,\Big] 
& \leq c_1 \e\Big[\, \Big\{ \, \int_{0}^{T} |h({X}_{\eta_{n}(s)}^{(n)})|^2d[L]^{T_m}_s \,\Big\}^{1/2}\, \Big]\\
& \leq c_1^2 \|h\|_{\infty}\e\Big[ \sup_{s\leq T\wedge T_m} |L_{s}|\,\Big],
\end{align*}
for some $c_1>0$. The right hand side above is bounded by $\lambda:= c_1^2 \|h\|_{\infty} \e\big[ \sup_{s\leq T} |L_{s}|\big]<\infty$ for all $m \in \n$. To take the limit as $m\rightarrow \infty$ in the above inequalities we note that 
\begin{gather*}
\e\Big[\,\sup_{t\leq T} \, \Big|\int_{0}^{t} h({X}_{\eta_{n}(s)}^{(n)})dL^{T_m}_s \Big| \,\Big] = \e\Big[\,\sup_{t\leq T\wedge {T_m}} \, \Big|\int_{0}^{t} h({X}_{\eta_{n}(s)}^{(n)})dL_s \Big| \,\Big]
\end{gather*}
and monotone convergence theorem can be applied.

To estimate the time integral and the Brownian integral we proceed similarly to Remark 3.2 of \cite{GyRa}, however we have to pay extra attention as $X^{(n)}$ is not continuous. Using left continuity of ${X}_{-}^{(n)}$, there exists a localizing sequence $(T_m)_{m\in \mathbb{N}}$ such that $|{X}_{-}^{(n)}|$ when stopped at $T_m$ is bounded and the Brownian integral is a martingale.
By applying the Burkholder-Davis-Gundy inequality, linear growth condition on $\sigma$ and Jensen's inequality, we obtain
	\begin{align*}
		\e\Big[\,\sup_{t\leq T} \, \Big| \int_{0}^{t} \sigma({X}_{\eta_{n}(s)}^{(n)}) dW_s^{T_m} \Big|		\Big]
		\leq
		c_0 \Big\{
		\e \Big[\int_{0}^{T\wedge T_m} \Big( 1+ \sup_{u\leq s}|{X}_{\eta_{n}(u)}^{(n)}|^2\,\,\Big)ds\Big]
		\Big\}^\frac{1}{2}.
	\end{align*}
Using the linear growth condition on $b$ and the fact that for each $m\in \mathbb{N}$, there exists a constant $C_m$ such that $\sup_{u\leq s\wedge T_m} |{X}_{\eta_n(u)}^{(n)}| \leq  \sup_{u< s\wedge T_m} |{X}_{u}^{(n)}|\leq C_m$, we obtain
	\begin{align}
\e \Big[\,	\sup_{t\leq T}|X_{t\wedge T_m}^{(n)}|\,\Big]
 	& \leq |x_0| +\lambda + KT
 	+ K\e\Big[ \int_{0}^{T\wedge T_m}\,\sup_{u\leq s} |{X}_{\eta_n(u)}^{(n)}|\,ds \, \Big]
 	+ c_0 \Big\{\e\Big[\int_{0}^{T\wedge T_m} \big( 1+ \sup_{u\leq s}|{X}_{\eta_{n}(u)}^{(n)}|^2\,\,\big)ds\Big]\Big\}^\frac{1}{2}\nonumber \\
 	& \leq C_{T,x_0}
 	+ K\e\Big[\int_{0}^{T\wedge T_m} \,\sup_{u< s} |{X}_{u}^{(n)} |\,ds \, \Big]
 	+ c_0 \Big\{\e\Big[\int_{0}^{T\wedge T_m} \sup_{u< s}|{X}_{u}^{(n)}|^2\,ds\,\Big]\Big\}^\frac{1}{2} 
 	< \infty\label{l2.1}
		\end{align}
%\textcolor{red}{PROBLEM! $\sup_{s< t} X_{s\wedge T_m} \neq \sup_{s< t\wedge T_m} X_{s}$ (take $T_m < t$ and suppose $X$ jump at $T_m$), the equality $\sup_{s\leq t} X_{s\wedge T_m} = \sup_{s \leq t\wedge T_m} X_{s}$ is true, check! The order which you prove this is not good. Be careful the process is not continuous!}
where $C_{T,x_0}:=|x_0|+\lambda+ KT + c_0 \sqrt{T}$. Using the fact that $X^{(n)}$ is a c\`	adl\`ag process and we replace $\sup_{u< s}|{X}_{u}^{(n)}|$ by $\sup_{u\leq s}|{X}_{u}^{(n)}|$ in the Lebesgue integral, equation \eqref{l2.1} can be estimated by
\begin{align*}
\e \Big[\,	\sup_{t\leq T}|X_{t\wedge T_m}^{(n)}|\,\Big]
 	& \leq C_{T,x_0} + K\e\Big[\int_{0}^{T} \,\sup_{u \leq s} |{X}_{u\wedge T_m}^{(n)} |\,ds \, \Big]
 	+ c_0\Big\{\e\Big[\int_{0}^{T} \sup_{u\leq s}|{X}_{u\wedge T_m}^{(n)}|^2\,ds\,\Big]\Big\}^\frac{1}{2}.
\end{align*}
Then it follows from Lemma 3.2 (i) of \cite{GyRa} with $p =1$, $q = 2$ and $V(t)=Z(t)=\sup_{u \leq t} |X_{u\wedge T_m}^{(n)}|$ that there exists $C_T$ such that
\begin{align*}
\e \Big[\,	\sup_{t\leq T}|X_{t\wedge T_m}^{(n)}|\,\Big]
& \leq C_{T,x_0} C_T.
\end{align*}
Hence the result follows from an application of the monotone convergence theorem.
%%Also from Lemma 3.2 (i) of \cite{GyRa} we know that the constant $K(T\wedge T_n)$ is increasing in the time parameter and the result follows from an application of the monotone convergence theorem.

To prove $(ii)$, we note that the coefficients $b$, $\sigma$ satisfies the linear growth condition and $h$ is bounded, then
	\begin{align*}
	|{X}_{t}^{(n)}-{X}_{\eta_{n}(t)}^{(n)}|
%	&\leq 
%	|b({X}_{\iota_{n}(t)}^{(n)})| |t-\eta_{n}(t)|
%	+|\sigma(X_{\eta_{n}(t)}^{(n)})| |W_{t}-W_{\eta_{n}(t)}|
%	+|h({X}_{\iota_{n}(t)}^{(n)}) |L_{t}-L_{\eta_{n}(t)}|\\
	& \leq
	K(1+|{X}_{\eta_{n}(t)}^{(n)}|)
	\left(|t-\eta_{n}(t)|
	+ |W_{t}-W_{\eta_{n}(t)}|
	\right)
	+ \|h\|_{\infty}|L_{t}-L_{\eta_{n}(t)}|.
	\end{align*}
From \eqref{EM_esti_2} and Assumption \eqref{Ass_1}-(i), we have
\begin{align*}
	\e[|{X}_{t}^{(n)}-{X}_{\eta_{n}(t)}^{(n)}|]
	&\leq M_1 \big(|t-\eta_{n}(t)| + |t-\eta_{n}(t)|^{1/2} + |t-\eta_{n}(t)|^{\zeta} \big),\\
	& \leq 3T M_1 	\left(\frac{1}{n}\right)^\frac{1}{2},
\end{align*}
where the constant $M_1$ is given by
\begin{align*}
	M_1
	:=\max\left\{K(1+C_1)(1 \vee \sqrt{2 \pi^{-1}}), \|h\|_{\infty}K_0 \right\}.
\end{align*}
%Finally, we use the assumption that $\mathbb{E}[|L_t-L_{s}|] \leq K_0|t-s|^\zeta$ for some $\frac{1}{2}\leq  \zeta \leq 1$ to obtain
%Finally, we use the Jensen inequality to obtain
%\begin{align*}
%&M_1 \left\{\mathbb{E}[|t-\eta_{n}(t)|] + \mathbb{E}[|t-\eta_{n}(t)|^{1/2}] + \mathbb{E}[|t-\eta_{n}(t)|^{\zeta}] \right\}\\
%&\leq M_1  \left\{ K_A\frac{1}{n}+ \left(K_A\frac{1}{n}\right)^\frac{1}{2}+\left(K_A\frac{1}{n}\right)^\zeta\,\right\}
%\leq 3 M_1 (K_A \vee 1) \left(\frac{1}{n}\right)^\frac{1}{2}.
%\end{align*}
This concludes the proof.
\end{proof}

%\begin{Rem}
%If the process $L$ is square integrable, then $h$ can also be taken to be of linear growth and there exists $C_2>0$ depend on $r$, $K$ and $T$ such that for any $t \in [0,T]$,
%	\begin{align*}
%	\e[|X_{t}^{(n)}-X_{\eta_{n}(t)}^{(n)}|^2]
%%	\leq C_2\left\{\mathbb{E}[|t-\eta_{n}(t)|] + \mathbb{E}[|t-\eta_{n}(t)|^2]\right\}
%	\leq \frac{C_2\log n}{n}.
%	\end{align*}	
%The assumption that $h$ is bounded can likely be relaxed if one considers specific class of L\'evy processes.
%\end{Rem}
%

From Theorem 2.2 in \cite{LiMy}, under Assumption \ref{Ass_1} (and the assumption that $\sigma$ and $h$ are bounded) there exists a unique strong solution to the SDE \eqref{SDE_1}. We now present our first result on the rate of convergence for  the Euler-Maruyama scheme.

\begin{Thm}\label{main_1}
Suppose that Assumption \ref{Ass_1} holds and $\sigma$, $h$ are bounded. Then there exists $C_3>0$ depending on $x_0$, $K$, $T$, $\rho$, $\gamma$, $\beta$, $\|\sigma\|_{\infty}$ and $\|h\|_{\infty}$ such that for any $\varepsilon \in (0,\frac{1}{1-\beta}-\alpha_{\nu})$,
	\begin{align*}
	\sup_{0 \leq t \leq T}\e[|X_{t}-X_{t}^{(n)}|]
	\leq C_3
	\left\{ \begin{array}{ll}
	\displaystyle
	n^{-\rho/2}
	+ n^{-\frac{\beta}{2} \left(1-\frac{1}{2\gamma}\right)}
	&\gamma \in (1/2,1],\,\alpha_{\nu}<\frac{2(1-\gamma)}{1-\beta},\\
	\displaystyle
	n^{-\rho/2}
	+ n^{-\frac{\beta}{2} \left(1-\frac{1}{2-(\alpha_{\nu}+\varepsilon)(1-\beta)}\right)}
	&\gamma \in (1/2,1],\,\alpha_{\nu}\geq \frac{2(1-\gamma)}{1-\beta},\\
	\displaystyle
	(\log n)^{-1}
	&\gamma=1/2.
	\end{array}\right.
	\end{align*}	
Moreover, if $\nu(dz)$ is defined by
		\begin{align}\label{stable_meas}
			\nu(dz)=\frac{\1_{(0,\infty)} (z) \mu(z)}{z^{1+\alpha}} dz,
		\end{align}
for some $\alpha \in (1,2)$ and bounded measurable function $\mu$ then the above $\varepsilon$ can be chosen as zero and $\alpha_{\nu}=\alpha$.
\end{Thm}
\begin{Rem}
We set $\alpha_*:=\sup\{\widehat{\alpha}>1;\int_1^\infty z^{\widehat \alpha}\nu(dz) < \infty\}$.
We point out that if $\gamma \in [\frac{1}{2}, \frac{\alpha_*}{2}]$ then the boundedness assumption on $\sigma$ can be removed. The rate of convergence can be retrieve by performing similar computations as in Theorem \ref{l2t} and we leave this to reader.
\end{Rem}
\begin{proof}
	Define ${Z}_t^{(n)}:=X_t-X_t^{(n)}$ and let $\varepsilon \in (0,1)$ and $\delta \in (1,\infty)$.
	By using \eqref{phi3} and It\^o's formula,
	\begin{align*}
	&|{Z}_t^{(n)}|
	\leq \varepsilon
	+\phi_{\delta,\varepsilon}({Z}_t^{(n)})
	=\varepsilon
	+{M}_t^{n,\delta,\varepsilon}
	+{I}_t^{n,\delta,\varepsilon}
	+{J}_t^{n,\delta,\varepsilon}
	+{K}_t^{n,\delta,\varepsilon},
	\end{align*}
	where we set
	\begin{align*}
	{M}_t^{n,\delta,\varepsilon}
	:=& \int_{0}^{t} \phi_{\delta,\varepsilon}'(Z_{s}^{(n)})\{\sigma(X_{s})-\sigma(X_{\eta_n(s)}^{(n)})\}dW_{s}\\
	&+\int_{0}^{t} \int_{0}^{\infty}
	\left\{
	\phi_{\delta,\varepsilon}(Z_{s-}^{(n)}+\{h(X_{s-})-h(X_{\eta_n(s)}^{(n)})\}z)-\phi_{\delta,\varepsilon}(Z_{s-}^{(n)})
	\right\}
	\widetilde{N}(ds,dz),\\
	{I}_t^{n,\delta,\varepsilon}
	:=&\int_{0}^{t} \phi_{\delta,\varepsilon}'(Z_{s}^{(n)})\{b(X_{s})-b(X_{\eta_n(s)}^{(n)})\}ds,
	\quad\\
	{J}_t^{n,\delta,\varepsilon}
	:=& \frac{1}{2} \int_{0}^{t} \phi_{\delta,\varepsilon}''(Z_{s}^{(n)}) |\sigma(X_{s})-\sigma(X_{\eta_n(s)}^{(n)})|^2ds,\\
	{K}_t^{n,\delta,\varepsilon}
	:=&
	\int_{0}^{t} \int_{0}^{\infty}
	\Big\{
	\phi_{\delta,\varepsilon}(Z_{s-}^{(n)}+\{h(X_{s-})-h(X_{\eta_n(s)}^{(n)})\}z)-\phi_{\delta,\varepsilon}(Z_{s-}^{(n)})\\
	&-\{h(X_{s-})-h(X_{\eta_n(s)}^{(n)})\}z \phi_{\delta,\varepsilon}'(Z_{s-}^{(n)})
	\Big\}
	\nu(dz)ds.
	\end{align*}

By localization arguments, we can take ${M}_t^{n,\delta,\varepsilon}$ to be a martingale and can be removed after taking the expectation. Therefore we only estimate the terms ${I}_{t}^{n,\delta,\varepsilon}$, ${J}_{t}^{n,\delta,\varepsilon}$ and ${K}_{t}^{n,\delta,\varepsilon}$.
The coefficient $b_1$ is Lipschitz continuous and $b_2$ is non-increasing, we have for $x, y \in \real$ with $x \neq y$,
%The coefficient $b_1$ is Lipschitz continuous and $b_2$ is $\rho$-H\"older continuos and non-increasing, then from \eqref{phi1} we have for $x, y \in \real$ with $x \neq y$,
	\begin{align*}
	\phi'_{\delta,\varepsilon}(x-y)(b(x)-b(y))	
	=\frac{\phi'_{\delta,\varepsilon}(x-y)}{x-y} (x-y)(b(x)-b(y))
	 \leq K_1\frac{|\phi'_{\delta,\varepsilon}(x-y)|}{|x-y|} |x-y|^2
	 \leq K|x-y|,
	\end{align*}
	where in the first inequality, we used \eqref{phi1} and the fact that $(x-y)(b_2(x)-b_2(y)) \leq 0$ and in the last inequality, we used \eqref{phi2} and Lipschitz continuity of $b_1$.
	From the above we have	
	\begin{align}\label{EP_6}
	I_t^{n,\delta,\varepsilon}
	&\leq \int_{0}^{t} \phi_{\delta,\varepsilon}'(Z_{s}^{(n)}) (b(X_{s})-b(X_{s}^{(n)}))ds
	+ \int_{0}^{t} \phi_{\delta,\varepsilon}'(Z_{s}^{(n)}) (b(X_{s}^{(n)})-b(X_{\eta_n(s)}^{(n)}))ds \notag\\
	&\leq K \int_{0}^{t} |Z_s^{(n)}| ds
	+K\int_{0}^{t} |X_{s}^{(n)}-X_{\eta_n(s)}^{(n)}|+|X_{s}^{(n)}-X_{\eta_n(s)}^{(n)}|^{\rho}ds.
	\end{align}
	Using the fact that $\sigma$ is bounded and \eqref{phi4}, we have
	\begin{align}
	J_t^{n,\delta,\varepsilon}
	&\leq \int_{0}^{t} \phi_{\delta,\varepsilon}''(Z_{s}^{(n)}) |\sigma(X_{s})-\sigma(X_{s}^{(n)})|^2ds
	+(2\|\sigma\|_{\infty})^{2-1/\gamma} \int_{0}^{t} \phi_{\delta,\varepsilon}''(Z_{s}^{(n)}) |\sigma(X_{s}^{(n)})-\sigma(X_{\eta_n(s)}^{(n)})|^{1/\gamma}ds \notag\\
	&\leq 2K^2\int_{0}^{t} \frac{\1_{[\varepsilon/\delta,\varepsilon]}(|Z_s^{(n)}|) |Z_s^{(n)}|^{2\gamma}}{|Z_s^{(n)}|\log \delta}ds
	+ 2 K^{1/\gamma} (2\|\sigma\|_{\infty})^{2-1/\gamma} \int_{0}^{t} \frac{\1_{[\varepsilon/\delta,\varepsilon]}(|Z_s^{(n)}|) |X_{s}^{(n)}-X_{\eta_n(s)}^{(n)}|}{|Z_s^{(n)}|\log \delta}ds \notag\\
	&\leq \frac{2TK^2\varepsilon^{2\gamma-1}}{\log \delta}
	+ \frac{2K^{1/\gamma}(2\|\sigma\|_{\infty})^{2-1/\gamma} \delta}{\varepsilon \log \delta}  \int_{0}^{t} |X_{s}^{(n)}-X_{\eta_n(s)}^{(n)}|ds.\label{EP_7}
	\end{align}	
	Finally, to estimate $K_t^{n,\delta,\varepsilon}$, we write it into two terms
	\begin{align*}
	K_t^{n,\delta,\varepsilon}
	=K_t^{n,\delta,\varepsilon,1}+K_t^{n,\delta,\varepsilon,2},
	\end{align*}
	where $K_t^{n,\delta,\varepsilon,1}$ and $K_t^{n,\delta,\varepsilon,2}$ are given by
	\begin{align*}
	K_t^{n,\delta,\varepsilon,1}
	&:=
	\int_{0}^{t} \int_{0}^{\infty}
	\Big\{
	\phi_{\delta,\varepsilon}(Z_{s}^{(n)}+\{h(X_{s})-h(X_{s}^{(n)})\}z)-\phi_{\delta,\varepsilon}(Z_{s}^{(n)})
	-\{h(X_{s})-h(X_{s}^{(n)})\}z \phi_{\delta,\varepsilon}'(Z_{s}^{(n)})
	\Big\}
	\nu(dz)ds\\
	K_t^{n,\delta,\varepsilon,2}
	&:=
	\int_{0}^{t} \int_{0}^{\infty}
	\Big\{
	\phi_{\delta,\varepsilon}(Z_{s}^{(n)}+\{h(X_{s})-h(X_{\eta_n(s)}^{(n)})\}z)-\phi_{\delta,\varepsilon}(Z_{s}^{(n)}+\{h(X_{s})-h(X_{s}^{(n)})\}z)\\
	& \quad -\{h(X_{s}^{(n)})-h(X_{\eta_n(s)}^{(n)})\}z \phi_{\delta,\varepsilon}'(Z_{s}^{(n)})
	\Big\}
	\nu(dz)ds.
	\end{align*}

%	We first consider $K_t^{n,\delta,\varepsilon,1}$.
	We observe that for each $s \in [0,t]$, if $Z_{s}^{(n)}=0$ then $h(X_{s})-h(X_{s}^{(n)})=0$. Therefore we can apply Lemma \ref{key_lem_0} with $y=Z_{s}^{(n)}$ and $x=h(X_{s})-h(X_{s}^{(n)})$ since $h$ is non-decreasing.	That is for any $u>0$, 
	\begin{align}
	&\int_{0}^{\infty}
	\left\{
	\phi_{\delta,\varepsilon}(Z_{s}^{(n)}+\{h(X_{s})-h(X_{s}^{(n)})\}z)-\phi_{\delta,\varepsilon}(Z_{s}^{(n)})
	-\{h(X_{s})-h(X_{s}^{(n)})\}z \phi_{\delta,\varepsilon}(Z_{s}^{(n)})
	\right\}
	\nu(dz) \notag\\
	&\leq
	\frac{2|h(X_{s})-h(X_{s}^{(n)})|^2 \1_{(0,\varepsilon]}(|Z_{s}^{(n)}|)}{|Z_{s}^{(n)}| \log \delta} \int_{0}^{u} z^2 \nu(dz)
	+ 2|h(X_{s})-h(X_{s}^{(n)})| \1_{(0,\varepsilon]}(|Z_{s}^{(n)}|) \int_{u}^{\infty} z \nu(dz)\notag\\
	& \leq \frac{2K^{2} |Z_{s}^{(n)}|^{2 \beta} \1_{(0,\varepsilon]}(|Z_{s}^{(n)}|)}{|Z_{s}^{(n)}| \log \delta} \int_{0}^{u} z^2 \nu(dz)
	+ 2K |Z_{s}^{(n)}|^{\beta} \1_{(0,\varepsilon]}(|Z_{s}^{(n)}|) \int_{u}^{\infty} z \nu(dz) \notag \\
	& \leq \frac{2K^{2}}{\log \delta} \varepsilon^{2 \beta-1} \int_{0}^{u} z^2 \nu(dz)
	+ 2K \varepsilon^{\beta} \int_{u}^{\infty} z \nu(dz),\label{EP_2}
	\end{align}
	where in the second last inequality, we used the fact that $h$ is a $\beta$-H\"older continuous function with $\beta \in (1-1/{\alpha_{\nu}},1)$.

	We recall that $\alpha_{\nu}=\inf\{\widehat{\alpha}>1; \lim_{x \to 0+} x^{\widehat{\alpha}-1} \int_{x}^{\infty}z \nu(dz)=0\}$.
	From Lemma 2.1 in \cite{LiMy}, we know that $\alpha_{\nu} \in [1,2]$ and for any $\widehat{\alpha}>\alpha_{\nu}$, $\lim_{x \to 0+} x^{\widehat{\alpha}-2} \int_{0}^{x} z^2 \nu(dz)=0$.
	%From (i) of Remark \ref{remark1.2} we know that $\alpha_{\nu} \in [1,2]$ and for any $\widehat{\alpha}>\alpha_{\nu}$, $\lim_{x \to 0+} x^{\widehat{\alpha}-2} \int_{0}^{x} z^2 \nu(dz)=0$.
	Also by the definition of $\alpha_{\nu}$, $\lim_{x \to 0+} x^{\widehat{\alpha}-1} \int_{x}^{\infty}z \nu(dz)=0$.
	Let $u=\varepsilon^{q}$ for some $q >0$, which we will choose later.
	Since $\beta \in (1-1/\alpha_{\nu},1)$, we can take $\widehat{\alpha}$ such that $\alpha_{\nu}<\widehat{\alpha}<\frac{1}{1-\beta}$.
	Then for sufficiently small $\varepsilon$, equation \eqref{EP_2} can be further bounded as follows
	\begin{align*}
	&\frac{2K^{2}}{\log \delta} \varepsilon^{2 \beta-1} \int_{0}^{\varepsilon^{q}} z^2 \nu(dz)
	+ 2K \varepsilon^{\beta} \int_{\varepsilon^{q}}^{\infty} z \nu(dz) \\
	& = \frac{K^{2}}{\log \delta} \varepsilon^{2 \beta-1-q(\widehat{\alpha}-2)} \varepsilon^{q(\widehat{\alpha}-2)}\int_{0}^{\varepsilon^{q}} z^2 \nu(dz)
	+ 2K \varepsilon^{\beta-q(\widehat{\alpha}-1)} \varepsilon^{q(\widehat{\alpha}-1)}\int_{\varepsilon^{q}}^{\infty} z \nu(dz) \\%\label{EP_3} \\
	&\leq \frac{2K^{2}}{\log \delta} \varepsilon^{2 \beta-1-q(\widehat{\alpha}-2)} 
	+ 2K \varepsilon^{\beta-q(\widehat{\alpha}-1)}
	=2\left(\frac{K^{2}}{\log \delta}+ K\right)\varepsilon^{1-\widehat{\alpha}(1-\beta)} \notag,
	\end{align*}
	where in the last equality, we have chosen $q>0$ such $2 \beta-1-q(\widehat{\alpha}-2)=\beta-q(\widehat{\alpha}-1)$, that is, $q=1-\beta$. From the above computation we have
	\begin{align}\label{EP_5}
	K_t^{n,\delta,\varepsilon,1}
	\leq 2T \left\{ \frac{K^{2}}{\log \delta}+ K\right\}\varepsilon^{1-\widehat{\alpha}(1-\beta)}.
	\end{align}		
	By applying \eqref{key_lem_122} in Lemma \ref{key_lem12} with $u=1,
		y=Z_{s}^{(n)},
		x=h(X_{s})-h(X_{\eta_n(s)}^{(n)})$,	$x'=h(X_{s})-h(X_{s}^{(n)})$
	and using the fact that $h$ is bounded, $K_t^{n,\delta,\varepsilon,2}$ can be bounded above by 
	\begin{align}\label{EP_8}
	&K_t^{n,\delta,\varepsilon,2}
	\leq |K_t^{n,\delta,\varepsilon,2}| \notag\\
	&\leq 2 \int_{0}^{1} z^2 \nu(dz) \int_{0}^{t}
	 \frac{\delta}{\varepsilon \log \delta} 
		\left(
			|h(X_{s}^{(n)})-h(X_{\eta_n(s)}^{(n)})|^2
			+|h(X_{s})-h(X_{s}^{(n)})| |h(X_{s}^{(n)})-h(X_{\eta_n(s)}^{(n)})|
		\right)
		\,ds \notag\\
	&\quad + 2 \int_{1}^{\infty} z \nu(dz) \int_{0}^{t} |h(X_{s}^{(n)})-h(X_{\eta_n(s)}^{(n)})| ds \nonumber \\
	&\leq 2
	\left\{
		\frac{4\|h\|_{\infty} \delta}{\varepsilon \log \delta}
		\int_{0}^{1} z^2 \nu(dz) \,ds
		+ \int_{1}^{\infty} z \nu(dz)
	\right\}
	\int_{0}^{t}
	|h(X_{s}^{(n)})-h(X_{\eta_n(s)}^{(n)})| ds \nonumber \\
	&\leq 2K
	\left\{
		\left( 4\|h\|_{\infty} \int_{0}^{1} z^2 \nu(dz)\right) \vee \int_{1}^{\infty} z \nu(dz)
	\right\}
	\left(\frac{\delta}{\varepsilon \log \delta} + 1\right)
	\int_{0}^{t} |X_{s}^{(n)}- X_{\eta_n(s)}^{(n)}|^{\beta} ds.
	\end{align}			

%	By using Taylor's expansion, the integrand of $K_t^{\delta,\varepsilon,2}$ equals to the following	
%	\begin{align}\label{EP_8}
%	|h(X_{s}^{(n)})-h(X_{\eta_n(s-)}^{(n)})|^2 z^2
%	\int_{0}^{1} 
%	\theta \phi_{\delta,\varepsilon}''(Z_s^{(n)}+\theta \{h(X_{s}^{(n)})-h(X_{\eta_n(s-)}^{(n)})\} )
%	d \theta.
%	\end{align}
%	Since $h$ is $\beta$-H\"older continuous, from \eqref{phi4}, \eqref{EP_8} is bounded by
%	\begin{align}\label{EP_9}
%	&\frac{2 K^2 |X_{s}^{(n)}-X_{\eta_n(s-)}^{(n)}|^{2\beta} z^2}{\log \delta}
%	\int_{0}^{1}
%	\frac{\theta \1_{[\varepsilon/\delta,\varepsilon]}(|Z_s^{(n)}+\theta \{h(X_{s}^{(n)})-h(X_{\eta_n(s-)}^{(n)})\}|) }{|Z_s^{(n)}+\theta \{h(X_{s}^{(n)})-h(X_{\eta_n(s-)}^{(n)})\}|}
%	d \theta \notag\\
%	&\leq 
%	\frac{K^2 |X_{s}^{(n)}-X_{\eta_n(s-)}^{(n)}|^{2\beta} \delta z^2}{\varepsilon \log \delta}.
%	\end{align}
%	Hence we obtain
%	\begin{align}\label{EP_10}
%	K_t^{\delta,\varepsilon,2}
%	&\leq \frac{K^2 \delta }{\varepsilon \log \delta} \int_{0}^{t} |X_{s}^{(n)}-X_{\eta_n(s-)}^{(n)}|^{2\beta} \int_{0}^{\infty} z^2\nu(dz) ds \notag\\
%	&=\frac{K_{\nu} K^2 \delta }{\varepsilon \log \delta} \int_{0}^{t} |X_{s}^{(n)}-X_{\eta_n(s-)}^{(n)}|^{2\beta} ds,
%	\end{align}
%	where $K_{\nu}:=\int_{0}^{\infty} z^2\nu(dz)$.
	
\noindent 	By taking the expectation in \eqref{EP_6}, \eqref{EP_7}, \eqref{EP_5} and \eqref{EP_8}, we obtain for any $t \in [0,T]$,
	\begin{align*}%\label{EP_11}
	\e[|Z_t^{(n)}|]
	& \leq \varepsilon
	+\e[{I}_t^{n,\delta,\varepsilon}]
	+\e[{J}_t^{n,\delta,\varepsilon}]
	+\e[K_t^{n,\delta,\varepsilon}] \notag\\
	&\leq \varepsilon 
	+K \int_{0}^{t} \e[|Z_{s-}^{(n)}|] ds
	+\frac{2TK^2\varepsilon^{2\gamma-1}}{\log \delta}
	+2T\left\{ \frac{K^{2}}{\log \delta}+ K\right\}\varepsilon^{1-\widehat{\alpha}(1-\beta)} \notag\\
	&\quad +K\int_{0}^{t} \e[|X_{s-}^{(n)}-X_{\eta_n(s)}^{(n)}|]+\e[|X_{s-}^{(n)}-X_{\eta_n(s)}^{(n)})|^{\rho}]ds\\
	&\quad +\frac{2K^{1/\gamma}(2\|\sigma\|_{\infty})^{2-1/\gamma}  \delta}{\varepsilon \log \delta}  \int_{0}^{t} \e[|X_{s-}^{(n)}-X_{\eta_n(s)}^{(n)})|] ds \notag \\
	&\quad+2K
	\left\{
	\left( 4\|h\|_{\infty} \int_{0}^{1} z^2 \nu(dz)\right) \vee \int_{1}^{\infty} z \nu(dz)
	\right\} \left( \frac{\delta}{\varepsilon \log \delta} + 1\right) \int_{0}^{t}
		\e[|X_{s}^{(n)}-X_{\eta_n(s)}^{(n)}|^{\beta}]ds.
	\end{align*}
	Using (ii) of Lemma \ref{EM_esti_1}, we have
	\begin{align*}%\label{EP_12}
	\e[|Z_t^{(n)}|]
	&\leq \varepsilon 
	+K \int_{0}^{t} \e[|Z_s^{(n)}|] ds
	+\frac{2TK^2\varepsilon^{2\gamma-1}}{\log \delta}
	+2T\left\{ \frac{K^{2}}{\log \delta}
	+K\right\}\varepsilon^{1-\widehat{\alpha}(1-\beta)} \notag\\
	&\quad+KT \left\{ \frac{C_2}{n^{1/2}}+\frac{C_2^{\rho} }{n^{\rho/2}} \right\}
	+2K^{1/\gamma}T (2\|\sigma\|_{\infty})^{2-1/\gamma} \frac{\delta}{\varepsilon \log \delta} \frac{C_2}{n^{1/2}} \notag\\
	&\quad+2KT
	\left\{
		\left( 4\|h\|_{\infty} \int_{0}^{1} z^2 \nu(dz)\right) \vee \int_{1}^{\infty} z \nu(dz)
	\right\} \left( \frac{\delta}{\varepsilon \log \delta} + 1\right) \frac{C_2^{\beta}}{n^{\beta/2}}.
	\end{align*}
	By using Gronwall's inequality, we have
	\begin{align*}%\label{EP_13}
 e^{-KT} \e[|Z_t^{(n)}|] 
 		&\leq
		\varepsilon
		+\frac{2TK^2\varepsilon^{2\gamma-1}}{\log \delta}
		+2T\left\{ \frac{K^{2}}{\log \delta}
		+K\right\}\varepsilon^{1-\widehat{\alpha}(1-\beta)} \notag\\
		&\quad+KT \left\{ \frac{C_2}{n^{1/2}}+\frac{C_2^{\rho} }{n^{\rho/2}} \right\}
		+2K^{1/\gamma}T (2\|\sigma\|_{\infty})^{2-1/\gamma}  \frac{\delta}{\varepsilon \log \delta} \frac{C_2}{n^{1/2}} \notag\\
		&\quad+2KT
		\left\{
		\left( 4\|h\|_{\infty} \int_{0}^{1} z^2 \nu(dz)\right) \vee \int_{1}^{\infty} z \nu(dz)
		\right\} \left( \frac{\delta}{\varepsilon \log \delta} + 1\right) \frac{C_2^{\beta}}{n^{\beta/2}}.
	\end{align*}
		
	To optimize the above bound, if $\gamma \in (1/2,1]$, then we choose $\delta =2$ and obtain
	\begin{align*}%\label{EM_14}
	\e[|Z_t^{(n)}|]
	\leq M_2
	\left\{
	\varepsilon
	+\varepsilon^{2\gamma-1}
	+\varepsilon^{1-\widehat{\alpha}(1-\beta)}
	+\frac{1}{n^{\rho/2}} 
	+\frac{1}{\varepsilon n^{1/2}}
	+\left(\frac{1}{\varepsilon}+1\right)\frac{1}{n^{\beta/2}}
	\right\},
	\end{align*}
	where the constant $M_2$ given by
	\begin{align*}
	M_2
	:=e^{KT} \max\Bigg\{
	1,\frac{2TK^2}{\log 2},T\left\{ \frac{K^{2}}{\log 2}+ K\right\},
	2KT\{C_2+C_2^{\rho}\},
	\frac{4 K^{1/\gamma}T (2\|\sigma\|_{\infty})^{2-1/\gamma} }{\log 2},\\
	2KT
	\left\{
	\left( 2\|h\|_{\infty} \int_{0}^{1} z^2 \nu(dz)\right) \vee \int_{1}^{\infty} z \nu(dz)
	\right\} \frac{2C_2^{\beta}}{\log 2} 
	\Bigg\}.
	\end{align*}
	We let $\varepsilon=n^{-q}$, where the optimal $q>0$ is chosen later. There are two cases to consider. If $\alpha_{\nu}<\frac{2(1-\gamma)}{1-\beta}$, then we choose $\widehat{\alpha}=\frac{2(1-\gamma)}{1-\beta}$ and we have $2\gamma-1=1-\widehat{\alpha}(1-\beta)$.
	Hence by choosing $q$ such $q(2\gamma-1)=\beta/2-q$, that is $q=\frac{\beta}{4\gamma}$, we have
	\begin{align*}%\label{EM_15_2}
	\e[|Z_t^{(n)}|]
	\leq 
	6M_2\left\{
	n^{-\rho/2}
	+ n^{-\frac{\beta}{2}\left(1-\frac{1}{2\gamma}\right)}
	\right\}.
	\end{align*}
	If $\alpha_{\nu} \geq \frac{2(1-\gamma)}{1-\beta}$, then we choose $\widehat{\alpha}=\alpha_{\nu}+\varepsilon$ for any $\varepsilon \in (0,\frac{1}{1-\beta}-\alpha_{\nu})$ and then $2\gamma-1>1-(\alpha_{\nu}+\varepsilon)(1-\beta)$.
	Hence by choosing $q$ such that  $q(1-(\alpha_{\nu}+\varepsilon)(1-\beta))=\beta/2-q$, that is $q=\frac{\beta}{2}\frac{1}{2-(\alpha_{\nu}+\varepsilon)(1-\beta)}$, we have
	\begin{align*}
	\e[|Z_t^{(n)}|]
	\leq 
	6M_2\left\{
		n^{-\rho/2}
		+ n^{-\frac{\beta}{2}\left(1-\frac{1}{2-(\alpha_{\nu}+\varepsilon)(1-\beta) }\right)}
	\right\}.
	\end{align*}
	This concludes the proof for $\gamma \in (1/2,1]$.

	If $\gamma = 1/2$, then we choose $\varepsilon=n^{-q}$ and $\delta =n^{p}$ with $p,q>0$ and $p+q<\beta/2$, we have
	\begin{align*}
%		\label{EM_16}
		e^{-KT} \e[|Z_t^{(n)}|]
		\leq &
		\frac{1}{n^{q}}
		+\frac{2TK^2}{p \log n}
		+2T\left\{ \frac{K^{2}}{p \log n}
		+K\right\}\frac{1}{n^{q-q\widehat{\alpha}(1-\beta)}} \notag\\
		&\quad+KT \left\{ \frac{C_2}{n^{1/2}}+\frac{C_2^{\rho} }{n^{\rho/2}} \right\}
		+K^{2}T \frac{n^{p+q}}{p \log n} \frac{C_2}{n^{1/2}} \notag\\
		&\quad+2KT
		\left\{
		\left( 4\|h\|_{\infty} \int_{0}^{1} z^2 \nu(dz)\right) \vee \int_{1}^{\infty} z \nu(dz)
		\right\} \left( \frac{n^{p+q}}{p \log n} + 1\right) \frac{C_2^{\beta}}{n^{\beta/2}}.
	\end{align*}
	Hence we can conclude that 
	\begin{align*}
	\e[|Z_t^{(n)}|]
	\leq \frac{M_3}{\log n},
	\end{align*}
	where the constant $M_3$ is given by 
	\begin{align*}
	M_3=
	e^{KT} \max\Bigg\{
	1,
	&\frac{2TK^2}{p},T\left\{\frac{K^2}{p}+K\right\},
	2KT\left\{C_2+C_2^{\rho} \right\},
	p^{-1}K^{2}C_2,\\
	&2KT
	\left\{
		\left( 4\|h\|_{\infty} \int_{0}^{1} z^2 \nu(dz)\right) \vee \int_{1}^{\infty} z \nu(dz)
	\right\} \left( p^{-1} + 1\right) C_2^{\beta}
	\Bigg\}.
	\end{align*}
	This concludes the proof for $\gamma =1/2$.
	
We consider now the L\'evy measure $\nu(dz)$ defined by
\begin{align*}
\nu(dz)=\frac{\1_{(0,\infty)} (z) \mu(z)}{z^{1+\alpha}} dz,
\end{align*}
for some $\alpha \in (1,2)$ and bounded measurable function $\mu$.
Then since
\begin{align*}
	\int_{x}^{\infty} z \nu(dz)
	\leq \|\mu\|_{\infty} \int_{x}^{\infty} z^{-\alpha} dz
	= \frac{\|\mu\|_{\infty} x^{1-\alpha}}{\alpha-1},
\end{align*}
we have $\alpha_{\nu}=\alpha$.
To conclude the statement, it is suffices to estimate the upper bounded of $K^{n,\delta,\varepsilon,1}$.
From \eqref{EP_2}, with $u=\varepsilon^{q}$ and $q >0$, we have
\begin{align*}
	K_t^{n,\delta,\varepsilon,1}
	&\leq \frac{2K^{2}T}{\log \delta} \varepsilon^{2 \beta-1} \int_{0}^{\varepsilon^{q}} z^2 \nu(dz)
	+ 2KT \varepsilon^{\beta} \int_{\varepsilon^{q}}^{\infty} z \nu(dz)  \notag\\
	&\leq  \frac{2K^{2} \|\mu\|_{\infty} }{\log \delta} \varepsilon^{2 \beta-1} \int_{0}^{\varepsilon^{q}} z^{1-\alpha} dz
	+ 2K\|\mu\|_{\infty} \varepsilon^{\beta} \int_{\varepsilon^{q}}^{\infty} z^{-\alpha}dz \notag \\
	&= \frac{2K^{2} \|\mu\|_{\infty}}{(2-\alpha) \log \delta} \varepsilon^{2 \beta-1-q(\alpha-2)} 
	+ \frac{2K\|\mu\|_{\infty}}{\alpha-1} \varepsilon^{\beta-q(\alpha-1)}\notag	 \\
	&=\left(\frac{2K^{2}}{(2-\alpha)\log \delta}+ \frac{2K}{(\alpha-1)}\right) \|\mu\|_{\infty} \varepsilon^{1-\alpha(1-\beta)} %\label{EP9},
\end{align*}
where in the last equality, we have chosen $q=1-\beta$.
This upper bound concludes the proof.
\end{proof}

\subsection{The Square Integrable Case}\label{L2}
In this subsection we compute the strong rate of convergence in the case where $L$ is a square integrable.
In this case, the boundedness condition on the coefficients $\sigma$ and $h$ can be lifted.
Examples of square integrable L\'evy process which can be simulated include compensated Poisson process, spectrally positive tempered stable processes or spectrally positive truncated stable processes.

%\begin{Thm}\label{l2t}
%Suppose that Assumption \ref{Ass_1} holds and $\int_{1}^{\infty}z^2 \nu(dz)<\infty$.
% Then there exists $C>0$ depend on $x_0$, $K$, $T$, $\rho$, $\gamma$ and $\beta$ such that for any $\varepsilon \in (0,\frac{1}{1-\beta}-\alpha_{\nu})$,
%\begin{align*}
%	\sup_{t \leq T}\e[|X_{t}-X_{t}^{(n)}|]
%	\leq C
%	\left\{ \begin{array}{ll}
%	\displaystyle
%	\left(\frac{1}{n}\right)^{\rho/2}
%	+ \left(\frac{1}{n}\right)^{\frac{\beta}{2} \left(1-\frac{1}{2\gamma}\right)}
%	&\gamma \in (1/2,1],\,\alpha_{\nu}<\frac{2(1-\gamma)}{1-\beta},\\
%	\displaystyle
%	\left(\frac{1}{n}\right)^{\rho/2}
%	+ \left(\frac{1}{n}\right)^{\frac{\beta}{2} \left(1-\frac{1}{2-(\alpha_{\nu}+\varepsilon)(1-\beta)}\right)}
%	& \gamma \in (1/2,1],\, \alpha_{\nu}\geq \frac{2(1-\gamma)}{1-\beta},\\
%	\displaystyle
%	\frac{1}{\log n}
%	& \gamma=1/2.
%	\end{array}\right.
%\end{align*}
%Moreover, if the L\'evy measure $\nu(dz)$ is defined by \eqref{stable_meas} for some $\alpha \in (1,2)$ and bounded measurable function $\mu$, then the above $\varepsilon$ can be chosen as zero and $\alpha_{\nu}=\alpha$.
%\end{Thm}
%%\begin{proof}
%%	We remove the hitting error $n^{-1/4}$ and replace $\log(n)/n$ by $1/n$ for the case $\gamma \in (1/2,1]$.
%%\end{proof}
%%Before computing the hitting and discretization erorr, we give the following some useful results.

\begin{Lem}\label{lem_moment}
	Suppose that Assumption \ref{Ass_1} holds and $\int_{1}^{\infty}z^2 \nu(dz)<\infty$. 
\begin{itemize}
	\item[(i)] Then there exists a constant $C_3>0$ such that
	\begin{align}
	\e\big[\,\sup_{t \leq T} |X^{(n)}_t|^2\,\big]
	&\leq C_3, \label{lem_moment_eq1}
%	\e\big[\,\sup_{t \leq T} |X^{(n)}_t|^2 \,\big|\, \mathcal{G}\big]
%	&\leq C_3.\label{lem_moment_eq2}\\
%%	\e[\max_{0\leq k \leq n} |X^{(n)}_{\tau_{k}^{(n)}}|^2]
%%	&\leq C_3,\\ \label{lem_moment_eq3}\\
%	\e\big[ \,\big| \e[\max_{ k \leq n} |X^{(n)}_{\tau_{k}^{(n)}}|^2 | \,\mathcal{G}] \big|^2 \big]
%	&\leq C_3. \label{lem_moment_eq4}
	\end{align}
	\item[(i)]  Then there exists a constant $C_4>0$ such that and for any $t \in [0,T]$,
	\begin{align}\label{EP_esti_3}
	\e\big[|{X}_t^{(n)}-{X}^{(n)}_{\eta_{n}(t)}|^2\big]
	\leq \frac{C_4}{n}.
	\end{align}
\end{itemize}
\begin{proof}
The proof is similar to Lemma \ref{EM_esti_1}. It is sufficient to apply It\^o's isometry and linear growth condition on the coefficients.
\end{proof}
\end{Lem}

\begin{Thm}\label{l2t}
	Suppose that Assumption \ref{Ass_1} holds  and $\int_{1}^{\infty}z^2 \nu(dz)<\infty$. Then there exists $C_{5}>0$ such that for any $\varepsilon \in (0,\frac{1}{1-\beta}-\alpha_{\nu})$,
	\begin{align*}
	\sup_{t \leq T}\e[|X_{t}-{X}_{t}^{(n)}|]
	\leq C_{5}
	\left\{ \begin{array}{ll}
	\displaystyle
	n^{-\rho/2}
	+ n^{-\frac{\beta}{2} \left(1-\frac{1}{2\gamma}\right)}
	&\gamma \in (1/2,1],\,\alpha_{\nu}<\frac{2(1-\gamma)}{1-\beta},\\
	\displaystyle
	n^{-\rho/2}
	+ n^{-\frac{\beta}{2} \left(1-\frac{1}{2-(\alpha_{\nu}+\varepsilon)(1-\beta)}\right)}
	&\gamma \in (1/2,1],\, \alpha_{\nu}\geq \frac{2(1-\gamma)}{1-\beta},\\
	\displaystyle
	(\log n)^{-1}
	&\gamma=1/2.
	\end{array}\right.
	\end{align*}
\end{Thm}

	\begin{proof}
	%The proof is the same as before, we need only to take $u = \infty$ in Lemma \ref{key_lem_0} and Lemma \ref{key_lem12} and replace $\beta/2$ by $\beta$ in \eqref{EM_14}.
The proof is similar to that of Theorem \ref{main_1}. We recall that ${Z}_t^{(n)}:=X_t-X_t^{(n)}$ and in the proof of Theorem \ref{main_1}, the boundedness of $\sigma$ and $h$ were only used in the estimation of ${J}_t^{n,\delta,\varepsilon}$ and ${K}_t^{n,\delta,\varepsilon,2}$. Therefore, we present here only the estimates of ${J}_t^{n,\delta,\varepsilon}$ and ${K}_t^{n,\delta,\varepsilon,2}$. 

	Using the fact that $\sigma$ is $\gamma$-H\"older continuous, we have
	\begin{align}\label{EP_L2_0}
	J_t^{n,\delta,\varepsilon}
	&\leq \int_{0}^{t} \phi_{\delta,\varepsilon}''(Z_{s}^{(n)}) |\sigma(X_{s})-\sigma(X_{s}^{(n)})|^2ds
	+\int_{0}^{t} \phi_{\delta,\varepsilon}''(Z_{s}^{(n)}) |\sigma(X_{s}^{(n)})-\sigma(X_{\eta_n(s)}^{(n)})|^{2}ds \notag\\
	&\leq 2K^2\int_{0}^{t} \frac{\1_{[\varepsilon/\delta,\varepsilon]}(|Z_s^{(n)}|) |Z_s^{(n)}|^{2\gamma}}{|Z_s^{(n)}|\log \delta}ds
	+ 2K^2 \int_{0}^{t} \frac{\1_{[\varepsilon/\delta,\varepsilon]}(|Z_s^{(n)}|) |X_{s}^{(n)}-X_{\eta_n(s)}^{(n)})|^{2\gamma}}{|Z_s^{(n)}|\log \delta}ds \notag\\
	&\leq \frac{2TK^2\varepsilon^{2\gamma-1}}{\log \delta}
	+ \frac{2K^2 \delta}{\varepsilon \log \delta}  \int_{0}^{t} |X_{s}^{(n)}-X_{\eta_n(s)}^{(n)})|^{2\gamma}ds.
	\end{align}	
	Next, we estimate the $K_t^{n,\delta,\varepsilon,2}$ term.	By applying \eqref{key_lem_122} in Lemma \ref{key_lem12} with
	\begin{align*}
	u=+\infty,~
	y=Z_{s}^{(n)},~
	x=h(X_{s})-h(X_{\eta_n(s)}^{(n)})\quad \text{and}\quad
	x'=h(X_{s})-h(X_{s}^{(n)}),
	\end{align*}
	the term $K_t^{n,\delta,\varepsilon,2}$ can be bounded above by (see Remark \ref{key_lem13}),
	\begin{align*}
	&{K}_t^{n,\delta,\varepsilon,2}
	\leq |K_t^{n,\delta,\varepsilon,2}| \notag\\
	&\leq 2 \int_{0}^{t}
	\frac{\delta}{\varepsilon \log \delta} 
	\left( |h(X_{s}^{(n)})-h(X_{\eta_n(s)}^{(n)})|^2+|h( X_{s})-h(X_{s}^{(n)})| |h(X_{s}^{(n)})-h(X_{\eta_n(s)}^{(n)})| \right)
	\int_{0}^{\infty} z^2 \nu(dz) \,ds.
	\end{align*}
	Hence by taking the expectation of both hand sides and using the H\"older inequality, we have
	\begin{align*}
%	\label{EP_L2_1}
	\e[K_t^{n,\delta,\varepsilon,2}] \notag
	&\leq \frac{2 \delta}{\varepsilon \log \delta} \int_{0}^{\infty} z^2 \nu(dz)
	\int_{0}^{t}
	\e[|h(X_{s}^{(n)})-h( X_{\eta_n(s)}^{(n)})|^2]ds \notag\\
	&\quad
	+\frac{2 \delta}{\varepsilon \log \delta} \int_{0}^{\infty} z^2 \nu(dz)
	\int_{0}^{t}\e[|h(X_{s})-h( X_{s}^{(n)})|^2]^{1/2} \e[|h(X_{s}^{(n)})-h(X_{\eta_n(s)}^{(n)})|^2]^{1/2}
	 \,ds.
	\end{align*}
	Next, by using the fact that $h$ is of linear growth and $\beta$-H\"older continuous,
	\begin{align}\label{EP_L2_2}
	\e[K_t^{n,\delta,\varepsilon,2}] & \leq \frac{2 K^2 \delta}{\varepsilon \log \delta} \int_{0}^{\infty} z^2 \nu(dz)
	\int_{0}^{t}
	\e[|X_{s}^{(n)}- X_{\eta_n(s)}^{(n)}|^{2\beta}]ds \notag\\
	&\quad
	+\frac{2 \cdot 3^{1/2} K^{3/2} \delta }{\varepsilon \log \delta} \int_{0}^{\infty} z^2 \nu(dz)
	\int_{0}^{t}\e[(4+|X_{s}|^2+|X_{s}^{(n)} |^2)]^{1/2} \e[|X_{s}^{(n)}-X_{\eta_n(s)}^{(n)}|^{2\beta}]^{1/2}
	\,ds \notag\\
	&\leq 2 K^2 T C_4^{\beta} \int_{0}^{\infty} z^2 \nu(dz)
	\frac{\delta}{\varepsilon \log \delta}
	\left(\frac{1}{n}\right)^{\beta}
	\notag\\
	&\quad
	+2 \cdot 3^{1/2} K^{3/2}T C_4^{\beta/2}\int_{0}^{\infty} z^2 \nu(dz)
	\left\{ 4+\sup_{t \leq T} \e[|X_{s}|^2]+C_3\right\}^{1/2}
	\frac{\delta}{\varepsilon \log \delta}
	\left(\frac{1}{n}\right)^{\beta/2}.
	\end{align}
	Take the expectation in \eqref{EP_6}, \eqref{EP_L2_0}, \eqref{EP_5} and \eqref{EP_L2_2}, we obtain from \eqref{EP_esti_3} and the Gronwall's inequality, for any $t \in [0,T]$,
	\begin{align*} %\label{EP_L2_3}
	e^{-KT}\e[|Z_t^{(n)}|] 
	&\leq \varepsilon 
	+\frac{2TK^2\varepsilon^{2\gamma-1}}{\log \delta}
	+2T\left\{ \frac{K^{2}}{\log \delta}+ K\right\}\varepsilon^{1-\widehat{\alpha}(1-\beta)}\\
	& \quad +KT \left\{
		\left( \frac{C_4}{n}\right)^{1/2}
		+\left( \frac{C_4}{n}\right)^{\rho/2}
		\right\}
	 \notag \\
	&\quad
	+K^2T C_4^{\gamma} \frac{ \delta}{\varepsilon \log \delta}
	\left(\frac{1}{n}\right)^{\gamma}
	+2 K^2 T C_4^{\beta} \int_{0}^{\infty} z^2 \nu(dz)
	\frac{\delta}{\varepsilon \log \delta}
	\left(\frac{1}{n}\right)^{\beta}
	\notag\\
	&\quad
	+2 \cdot 3^{1/2} K^{3/2}T C_4^{\beta/2}\int_{0}^{\infty} z^2 \nu(dz)
	\left\{ 4+\sup_{ t \leq T} \e[|X_{s}|^2]+C_3\right\}^{1/2}
	\frac{\delta}{\varepsilon \log \delta}
	\left(\frac{1}{n}\right)^{\beta/2}.
	\end{align*}
		
	To optimize the above bound, if $\gamma \in (1/2,1]$, then we choose $\delta =2$ and obtain
	\begin{align*}
%	\label{EP_L2_4}
	\e[|Z_t^{(n)}|]
	\leq M_4
	\Bigg\{
	\varepsilon
	+\varepsilon^{2\gamma-1}
	+\varepsilon^{1-\widehat{\alpha}(1-\beta)}
	+\frac{1}{n^{\rho/2}}
	+\frac{1}{\varepsilon n^{\gamma}}
	+\frac{1}{\varepsilon n^{\beta}} 
	+\left(\frac{1}{\varepsilon}+1\right)\frac{1}{n^{\beta/2}}
	\Bigg\},
	\end{align*}
	where $M_4$ is some constant defined by
	\begin{align*}
	M_4
	:=e^{KT} \max\Bigg\{
	1,
	\frac{2TK^2}{\log 2},
	2T\left\{ \frac{K^{2}}{\log 2}+ K\right\},
	KT\{C_4^{1/2}+C_4^{\rho/2}\},
	\frac{2K^2T C_4^{\gamma}}{\log 2},
	\frac{4K^2 T C_4^{\beta}}{\log 2} \int_{0}^{\infty} z^2 \nu(dz),
	\\
	\frac{4 \cdot 3^{1/2} K^{3/2}T C_4^{\beta/2}}{\log 2} \int_{0}^{\infty} z^2 \nu(dz)
	\left\{ 4+\sup_{ t \leq T} \e[|X_{s}|^2]+C_3\right\}^{1/2}
	\Bigg\}.
	\end{align*}
	We choose $\varepsilon=n^{-q}$ and then we choose the optimal $q>0$.
	There are again two cases to consider, if $\alpha_{\nu}<\frac{2(1-\gamma)}{1-\beta}$, then we choose $\widehat{\alpha}=\frac{2(1-\gamma)}{1-\beta}$ and then $2\gamma-1=1-\widehat{\alpha}(1-\beta)$.
	Hence by choosing $q$ as $q(2\gamma-1)=\beta/2-q$, that is $q=\frac{\beta}{4\gamma}$, we have
	\begin{align*} %\label{EP_L2_6}
	\e[|Z_t^{(n)}|]
	\leq 
	7M_4\left\{
	\left(\frac{1}{n}\right)^{\rho/2}
	+ \left(\frac{1}{n}\right)^{\frac{\beta}{2}\left(1-\frac{1}{2\gamma}\right)}
	\right\}.
	\end{align*}
	If $\alpha_{\nu} \geq \frac{2(1-\gamma)}{1-\beta}$, then we choose $\widehat{\alpha}=\alpha_{\nu}+\varepsilon$ for any $\varepsilon \in (0,\frac{1}{1-\beta}-\alpha_{\nu})$ and then $2\gamma-1>1-(\alpha_{\nu}+\varepsilon)(1-\beta)$.
	Hence by choosing $q$ such that $q\{1-(\alpha_{\nu}+\varepsilon)(1-\beta)\}=\beta/2-q$, that is $q=\frac{\beta}{2}\frac{1}{2-(\alpha_{\nu}+\varepsilon)(1-\beta)}$, we have
	\begin{align*}
	\e[|Z_t^{(n)}|]
	\leq 
	7M_4\left\{
	\left(\frac{1}{n}\right)^{\rho/2}
	+ \left(\frac{1}{n}\right)^{\frac{\beta}{2} \left( 1-\frac{1}{2-(\alpha_{\nu}+\varepsilon)(1-\beta)}\right) }
	\right\}.
	\end{align*}
	This concludes the proof for $\gamma \in (1/2,1]$.
	
	If $\gamma = 1/2$, then we choose $\varepsilon=n^{-q}$ and $\delta =n^{p}$ with $p,q>0$ and $p+q<\beta/2<1/2=\gamma$. Then 
	\begin{align*}%\label{EP_L2_7}
	e^{-KT}\e[|Z_t^{(n)}|] 
	&\leq \frac{1}{n^{q}} 
	+\frac{2TK^2}{p \log n}
	+2T\left\{ \frac{K^{2}}{p \log n}+ K\right\} \frac{1}{n^{q-q\widehat{\alpha}(1-\beta)}}
	+KT \left\{
	\left( \frac{C_4}{n}\right)^{1/2}
	+\left( \frac{C_4}{n}\right)^{\rho/2}
	\right\}
	\notag \\
	&\quad
	+K^2T C_4^{1/2} \frac{ n^{p+q}}{p \log n}
	\left(\frac{1}{n}\right)^{1/2}
	+2 K^2 T C_4^{\beta} \int_{0}^{\infty} z^2 \nu(dz)
	\frac{n^{p+q}}{p\log n}
	\left(\frac{1}{n}\right)^{\beta}
	\notag\\
	&\quad
	+2 \cdot 3^{1/2} K^{3/2}T \int_{0}^{\infty} z^2 \nu(dz)
	\left\{ 4+\sup_{s \leq t} \e[|X_{s}|^2]+C_3\right\}^{1/2}
	\frac{n^{p+q}}{p\log n}
	\left(\frac{1}{n}\right)^{\beta/2}.
	\end{align*}
	Hence we can conclude that
	\begin{align*}
		\e[|Z_t^{(n)}|]
		\leq \frac{M_5}{\log n},
	\end{align*}
	where the constant $M_5$ is given by 
	\begin{align*}
	M_5=
	e^{KT} \max\Bigg\{
	1,
	\frac{2TK^2}{p},T\left\{\frac{K^2}{p}+K\right\},
	2KT\left\{C_4^{1/2}+C_4^{\rho/2} \right\},
	\frac{K^2T C_4^{1/2}}{p},
	\frac{2 K^2 T C_4^{\beta}}{p} \int_{0}^{\infty} z^2 \nu(dz),\\
	\frac{2 \cdot 3^{1/2} K^{3/2}TC_4^{\beta/2}}{p} \int_{0}^{\infty} z^2 \nu(dz)
	\left\{ 4+\sup_{s \leq t} \e[|X_{s}|^2]+C_3\right\}^{1/2}
	\Bigg\}.
	\end{align*}
	This concludes the proof for $\gamma =1/2$.
	
	\end{proof}

\section*{Acknowledgements}
The second author was supported by JSPS KAKENHI Grant Number 16J00894 and 17H06833.

\end{document}